\documentclass{article}
\usepackage{authblk}
\usepackage{amsmath}
\usepackage{amssymb}
\usepackage{amsthm}
\usepackage{graphicx}
\usepackage{xcolor}
\usepackage{enumitem}
\usepackage{hyperref}
\usepackage{cite}
\newcommand{\be}{\begin{equation}}

\newcommand{\ee}{\end{equation}}

\newcommand{\bea}{\begin{eqnarray}}

\newcommand{\eea}{\end{eqnarray}}

\newcommand{\beastar}{\begin{eqnarray*}}

\newcommand{\eeastar}{\end{eqnarray*}}

\newcommand{\ai}{A}
\newcommand{\bb}{B}
\newcommand{\cc}{C}
\newcommand{\dd}{D}
\newcommand{\eee}{E}
\newcommand{\ff}{F}
\newcommand{\ii}{I}
\newcommand{\jj}{J}
\newcommand{\kk}{K}
\newcommand{\li}{L}

\newtheorem{theorem}{Theorem}[section]
\newtheorem{lemma}[theorem]{Lemma}

\newtheorem{corollary}[theorem]{Corollary}

\theoremstyle{definition}
\newtheorem{definition}[theorem]{Definition}

\theoremstyle{remark}
\newtheorem{remark}[theorem]{Remark}

\title{Potential Carroll Structures and Special Carrollian Manifolds}

\author[1]{Samuel Blitz \thanks{blitz@math.muni.cz}}

\affil[1]{Department of 
Mathematics and Statistics, Masaryk University, Brno, Czech Republic}
\author[2]{Gabriel Herczeg \thanks{gabe.herczeg@gmail.com}}
\affil[2]{Brown University, Providence, RI, 02912, USA  }
\author[3]{David McNutt \thanks{dmcnutt@mil.no}}
\affil[3]{The Royal Norwegian Naval Academy, N-5165 Bergen, Norway }

\date{\today}

\begin{document}

\maketitle

\begin{abstract}
It is well-known that unlike space-like and time-like hypersurfaces, null hypersurfaces in Lorentzian manifolds do not naturally inherit an affine connection from the spacetime in which they are embedded. On the other hand, recent developments in flat-space holography motivate the study of the intrinsic geometry of null hypersurfaces such as null infinity and black hole event horizons. Here we initiate the study of \emph{potential Carroll structures}, a candidate for an intrinsic description of null hypersurfaces which may be particularly useful in settings where conformal isometries are of interest, and we explore their relationship to another such candidate intrinsic geometry, the special Carrollian manifolds.
\end{abstract}

\tableofcontents

\bigskip

\section{Introduction}
The holographic principle proposes that quantum gravitational information can be encoded in a quantum field theory defined on a lower-dimensional boundary of spacetime\cite{tHooft:1993dmi,Susskind:1994vu, Bousso:2002ju}.
While the best-studied realization of the holographic principle is the AdS/CFT correspondence
\cite{Maldacena:1997re,Gubser:1998bc,Witten:1998qj}, significant progress has been made over the last few years by applying holographic techniques to the asymptotically flat setting using two closely related frameworks: celestial holography \cite{He:2014laa, He:2014cra, Strominger:2014pwa, Pasterski:2015zua, Raclariu:2021zjz} and Carrollian holography \cite{Duval:2014uva, Donnay:2022wvx, Donnay:2022aba, Bagchi:2022emh}.  As the name suggests, celestial holography involves extracting quantum-gravitational data from  a  boundary conformal field theory defined on the celestial sphere, having co-dimension two relative to the  bulk spacetime. By contrast, Carrollian holography places the boundary conformal field theory on the whole of null infinity, having codimension one.

In light of these recent advances, physicists and mathematicians have turned to the study of the intrinsic geometry of null hypersurfaces, which arise naturally as both event horizons of black holes and the conformal boundary of asymptotically flat spacetimes, with renewed interest. As we (traditionally) occupy a four-dimensional spacetime, three-dimensional null hypersurfaces are of particular physical relevance. Indeed, it is perhaps not a coincidence that a variety of geometric structures that enjoy applications in mathematical physics can be imposed on three-manifolds that come equipped with a preferred direction. Among these are contact structures, Carrollian structures, CR structures, and a new structure discussed in this manuscript, the potential Carroll structure. Our primary focus will be on the relations between potential Carroll structures and a particular type of Carrollian structure---the special Carrollian manifold (SCM).

Manifolds with null metrics that are equipped with additional structure (such as an affine connection, a fundamental vector field, and/or an Ehresmann connection), are known as \textit{Carrollian} geometries~\cite{Duval:2014uva}. A \textit{Carrollian structure} is the minimal such geometry, being merely the triple $(M,g,\ell)$ where $M$ is a smooth manifold equipped with a null metric $g$ (with one-dimensional null space) and a fundamental vector field $\ell$ such that $g(\ell,\cdot) = 0$. While there is some contention regarding precisely \textit{how} one should further equip a Carrollian structure with a connection, one method is to pick out a null vector field $\ell$ along the hypersurface and a dual 1-form $\nu$ (called an Ehresmann connection), and then uniquely fix an affine connection $\nabla$ by demanding that it holds the degenerate metric $g$, the null vector field $\ell$, and the 1-form $\nu$ parallel and has ``minimal" torsion in a sense that will described in Definition~\ref{minimal-torsion-definition}. Such a structure $(M,g,\ell, \nu, \nabla)$ is called a special Carrollian manifold~\cite{BlitzMcNutt}; see Definition~\ref{def:SCM}. Note that another method of fixing the connection is by inducing the \textit{rigged} connection from the embedding spacetime~\cite{MarsSenovilla}, but we will not make use of this construction here.

On the other hand, given a manifold, a vector field, and a dual 1-form, other structures can be built beyond what is considered Carrollian. One prime example (when the manifold is odd-dimensional) is a contact structure, defined by the stipulation that the kernel of the 1-form is maximally non-integrable, in which case it is natural to select the vector field to encode the Reeb dynamics of the contact form. Alternatively, one may stipulate that the dual 1-form is a covariant potential for the (parallel) degenerate metric, that the fundamental vector field is parallel, and that the torsion of the connection is ``potential" in a way that will be detailed in Definition~\ref{minimal-torsion-definition}; this will be called a potential Carroll structure. In Definition~\ref{def:PCS} we will give a more thorough definition of this geometric structure.

A potential Carroll structure differs from a special Carrollian manifold in one critical way: the 1-form acts as a potential for the degenerate metric rather than being parallel. One of the primary goals of this manuscript is to establish conditions on when these two different Carrollian structures may be (uniquely) constructed from smaller data sets, and when these two different Carrollian structures may be transformed into one another. While some of the results proved here apply to manifolds of arbitrary dimension,  we will mostly focus on the 3-dimensional case where proofs are readily available, and where our results enjoy the most direct physical applications.

The structure of the paper is as follows. In Subsection~\ref{sec:notations}, we review the notation that will be used throughout this paper. In Section~\ref{sec:carollian}, we describe some geometric features of Carrollian structures which include both special Carrollian manifolds and potential Carroll structures. In Section~\ref{sec:SCMs} we construct the minimal data required to determine a special Carrollian manifold, and in Section~\ref{sec:PCSs}, we do the same for potential Carroll structures. Finally, in Section~\ref{sec:relations} we determine the conditions required to relate special Carrollian manifolds and potential Carroll structures.

\subsection{Notations and Conventions} \label{sec:notations}
We denote by $M^d$ a $d$-dimensional smooth manifold. We equip such a manifold with a degenerate metric $g$ with signature $(0, +, \ldots, +)$. While such a metric does not uniquely determine a connection, we may equip this manifold with an affine connection $\nabla$. For any connection $\nabla$, we denote the curvature tensor
$$R(X,Y)Z = (\nabla_X \nabla_Y - \nabla_Y \nabla_X - \nabla_{[X,Y]})Z\,,$$
where $X,Y,Z \in \Gamma(TM)$ and $[\cdot, \cdot]$ is the Lie bracket of vector fields. Similarly, the torsion is defined according to
$$T(X,Y) = \nabla_X Y - \nabla_Y X - [X,Y]\,.$$
When writing equations in this index-free notation, we will often $\odot$ as a binary operation indicating (unit normalized) symmetrization of the two tensors surrounding the operation.

Frequently, we will rely on Penrose's abstract index notation, using letters at the beginning of the latin alphabet to denote tensor products of the (co)tangent bundle on $M$, i.e. we may write $t^{ab}{}_{c} \in \Gamma(TM \otimes TM \otimes T^*M)$. In this notation, we use round brackets $(\,)$ to denote (unit normalized) symmetrization, i.e. $t_{(ab)} = \tfrac{1}{2}(t_{ab} + t_{ba})$ and square brackets $[\, ]$ to denote (unit normalized) antisymmetrization, i.e. $t_{[ab]} = \tfrac{1}{2}(t_{ab} - t_{ba})$. 

In this language, we may express the covariant derivative in terms of partial derivatives $\partial$ and connection coefficients $\Gamma$, so that for a vector field $V \in \Gamma(TM)$,
$$\nabla_a V^b = \partial_a V^b + \Gamma^b{}_{ca} V^c\,.$$ 
Furthermore, the curvature tensor may be expressed as
$$X^a Y^b R_{ab}{}^c{}_d Z^d = X^a Y^b (\nabla_a \nabla_b - \nabla_b \nabla_a) Z^c\,,$$
and the torsion tensor may be expressed as
$$T^c{}_{ab} = -2 \Gamma^c{}_{[ab]}\,.$$

On occasion, we will require derivatives and curvatures on a distinguished hypersurface embedded in $M$. In that case, we will use overbars $\bar{\bullet}$ to denote the quantities intrinsic to that distinguished hypersurface. For example, for an embedding $\Sigma \hookrightarrow (M,g)$, we will denote by $\bar{\nabla}$ a connection (perhaps induced from $\nabla$) on $T \Sigma$. Similarly, $\bar{R}$ will denote the curvature of that connection, and, for example, $\bar{Sc}$ will denote its scalar curvature (if such a quantity is available).

\bigskip

At times it will be helpful to work with non-coordinate bases. In order to do so, we must choose a 1-form, $\nu$ such that $\nu(\ell) = 1$ where $\ell$ is the fundamental vector field.  A basis for $T^*M$, known as a coframe and denoted as $(\theta^\ai)$ , is given by
\begin{equation}
    \theta^1 = \nu, \;\;\theta^{\ii} = m^{\ii}
\end{equation}
\noindent where $m^\ii$ diagonalizes the degenerate metric, i.e., $g = \delta_{\ii\jj} m^\ii m^\jj$. From the coframe we can determine a basis for $TM$, which is known as a frame and is denoted as $(e_\ai)$, by taking $e_1 = \ell$ and requiring that the natural pairing of 1-forms with vector fields satisfies: 
\begin{equation}
    \langle \theta^\ai, e_\bb \rangle = \delta^\ai_{~\bb}. 
\end{equation}
\noindent This is sufficient to determine the remaining frame elements, $e_\ii$, uniquely. We note that since the frame is a non-coordinate basis, the Lie brackets of the frame elements are non-zero:
\begin{equation}
    [e_\ai, e_\bb] = \hat{C}^\cc_{\ai \bb} e_\cc.     
\end{equation} 
With a coframe and frame, the connection can be determined from 
\begin{equation}
    \nabla_\ai e_\bb = \Gamma^\cc_{~\bb \ai} e_\cc \quad \nabla_\ai \theta^\cc = \Gamma^\cc_{~\bb \ai} \theta^\bb. 
\end{equation}
\noindent Denoting the derivative of a function from the manifold to $R^{d}$, $f$ with respect to a frame element $e_A$ as $e_A(f)$, the torsion tensor and the curvature tensor can be defined as 
   \begin{equation}
        \begin{aligned}
            T^\ai_{~\bb\cc} &= -\Gamma^\ai_{~\bb \cc} + \Gamma^\ai_{~\cc \bb} - \hat{C}^\ai_{~\bb \cc}, \\
            R^\ai_{~\bb \cc \dd} &= e_\cc(\Gamma^\ai_{~\bb \dd}) - e_\dd(\Gamma^\ai_{\bb \cc}) + \Gamma^\eee_{~\bb\dd} \Gamma^\ai_{~\eee\cc} - \Gamma^\eee_{~\bb\cc} \Gamma^\ai_{~\eee\dd} - \hat{C}^\eee_{~\cc\dd} \Gamma^A_{~\bb \eee},
        \end{aligned} \label{eqn:GeomTensors}
    \end{equation}

\section{Carrollian Structures, Connections, and Torsions} \label{sec:carollian}

One feature of any \textit{Carrollian manifold} (a Carrollian structure equipped with a compatible affine connection) is that its torsion tensor is heavily constrained by the failure of the fundamental vector field to be Killing. This feature is captured by the following lemma (first proven in a more algebraic setting in~\cite{figueroa2020intrinsic})
which will be critical in the following sections.

\begin{lemma} \label{thm:CarrollianConnection}
    Let $(M,g, \ell)$ be a Carrollian structure equipped with an affine connection $\nabla$ satisfying $\nabla g = \nabla \ell = 0$. Then the torsion tensor, $T$, associated with the connection satisfies: 
    \begin{align} \label{uniqueTorsion}
        (\mathcal{L}_{\ell} g)(X,Y) = g(T(\ell,X), Y)  + g( X, T(\ell,Y)), ~~\forall X,Y \in \Gamma(T M).
    \end{align}
\end{lemma}

\begin{proof}
    Since $\nabla g = 0$, for any vector-fields $X$ and $Y$ it follows that
    \begin{equation}
        \nabla_{\ell} (g(X,Y)) = \ell (g(X,Y)) - g (\nabla_{\ell} X, Y) - g(X,\nabla_{\ell} Y)=0. 
    \end{equation}
    Expanding the first term in Lie derivatives then gives
    \begin{equation}
        [\mathcal{L}_{\ell}g](X,Y)+g([\ell, X], Y)+ g(X, [\ell, Y]) - g (\nabla_{\ell} X, Y) - g(X,\nabla_{\ell} Y)=0.
    \end{equation}
    Now, taking $\nabla \ell =0$, the torsion definition gives
    \begin{equation}
        T(\ell, X) = \nabla_{\ell} X - \nabla_X \ell - [\ell, X] = \nabla_{\ell} X - [\ell, X], \nonumber
    \end{equation}
    \noindent for any vector-field $X$. Using this expression gives
    \begin{equation}
        [\mathcal{L}_{\ell}g](X,Y) - g(T(\ell, X), Y) - g(X, T(\ell, Y)) =0. 
    \end{equation}
    As $X$ and $Y$ are arbitrary, the required expression holds.
\end{proof}

That any such connection obeys this particularly simple geometric identity suggests a useful classification of tensors of this type. To that end, we introduce the following definition:
\begin{definition} \label{minimal-torsion-definition}
    Let $(M,g,\ell)$ be a Carrollian structure equipped with an Ehresmann connection $\omega$ such that $\omega(\ell) = 1$. Suppose $T \in \Gamma(TM \otimes \wedge^2 T^* M)$ is a section satisfying
        \begin{align}
        g(T(\ell,X), Y) &= \tfrac{1}{2} (\mathcal{L}_{\ell} g)(X,Y), ~~\forall X,Y \in \Gamma(T M)\,, \label{lie-g-constraint}\\
        \omega(T(\ell,X)) &= (\mathcal{L}_{\ell} \omega)(X) \,, \forall X,Y \in \Gamma(TM) \label{minimal-constraint} \\
        T(U,V) &= 0,~~\forall U,V \in \ker(\omega)\label{horiz-constraint}\,.
    \end{align}
    Then we say that the section $T$ is \textit{minimal}.
\end{definition}

From this definition, we are able to establish that minimal sections are uniquely determined by the tuple $(M,g,\ell,\omega)$:

\begin{lemma} \label{minimal-lemma}
    Let $(M^d,g,\ell)$ be a Carrollian structure equipped with an Ehresmann connection $\omega$ such that $\omega(\ell) = 1$. Then, there exists a unique minimal section $T \in \Gamma(TM \otimes \wedge^2 T^* M)$ associated with the tuple $(M,g,\ell,\omega)$.
\end{lemma}
\begin{proof}
    We prove by direct computation. First, we fix a frame basis $(\ell,e_{\ii})$ and a coframe basis $(\omega, \theta^\ii)$, where $\omega(e_\ii) = 0$, $\theta^\ii(\ell) = 0$, $\theta^\ii(e_\jj) = \delta^\ii_\jj$, and $g= \delta_{\ii\jj} \theta^\ii \theta^\jj$. Note that $\ai,\bb,\cc \in \{0, 1, \ldots, d\}$ and $\ii,\jj,\kk \in \{1, \ldots, d\}$.

    From Constraints~\ref{horiz-constraint} and~\ref{minimal-constraint},
    it is easy to see that both $T^\ai_{~\ii\jj}$ and $T^0_{~0\bb}$ are fixed, respectively.
    Now restricting Constraint~\ref{lie-g-constraint} to horizontal vectors, we have
    $$\tfrac{1}{2} (\mathcal{L}_{\ell} g)(U,V) = g(T(\ell,U), V) , ~~\forall U,V \in \ker (\omega)\,,$$
    one may see that $T^\kk_{~0\ii} g_{\jj\kk}$ is fixed.
    Thus, all of the components of $T$ are fixed by minimality.
\end{proof}

We now have the language to explicitly define special Carrollian manifolds and potential Carrollian structures:
\begin{definition}
    \label{def:SCM}
    Let $(M,g,\ell)$ be a Carrollian structure equipped with the following:
    \begin{itemize}
        \item A principal 1-form $\nu$ satisfying $\alpha(\nu) = 1$ and
        \item a connection $\nabla$ satisfying $\nabla \ell = 0$, $\nabla g = 0$, $\nabla \nu = 0$, and whose torsion is minimal.
    \end{itemize}
    Then the tuple $(M,g,\ell,\nu,\nabla)$ is a \textit{special Carrollian manifold}.
\end{definition}
\begin{definition} \label{def:PCS}
Let $(M,g,\ell)$ be a Carrollian structure equipped with the following:
\begin{itemize}
    \item A 1-form $\alpha$ satisfying $\alpha(\ell) = 1$ and 
    \item a connection $\nabla$ satisfying $\nabla \ell = 0$, $\nabla g = 0$, and $ \nabla \odot \alpha = g$
    and whose torsion is minimal.
\end{itemize}
Then, the tuple $(M,g,\ell,\alpha,\nabla)$ is a \textit{potential Carroll structure}.
\end{definition}

Having established these specializations of a Carrollian structure, we are now ready to study the conditions required to uniquely specify special Carrollian manifolds (Section~\ref{sec:SCMs}) and potential Carroll structures (Section~\ref{sec:PCSs}). As the geometric structures are similar, it is instructive to first establish a common lemma that will be useful later.

\begin{lemma} \label{nonvanishing-torsion-lemma}
    Let $(M^3,g,\ell)$ be a Carrollian structure equipped with a linear connection $\nabla$ satisfying $\nabla \ell = 0 = \nabla g$.
    Suppose that there exists an Ehresmann connection $\omega$ such that the torsion of $\nabla$ is minimal with respect to the tuple $(M,g,\ell,\omega)$. Assuming the torsion is non-vanishing, there exists a symmetric tensor $N$ such that
    $$\nabla_{(a} \omega_{b)} = N_{ab}\qquad \text{ where } \qquad N(\ell, \cdot) = 0$$
    if and only if one of the following conditions holds:
        \begin{itemize}
        \item for $V := Trace(T)$,  $V(\ell) \neq 0$, 
        $$\nabla_{(a}\gamma_{b)} = N_{ab}, \qquad \gamma_b = \frac{V_b-\mathcal{L}_\ell \omega_b}{V(\ell)} \,, $$
        \item $V(\ell) = 0$  
        $$-\nabla_d T_{abc} = \nabla_d (\mathcal{L}_\ell g_{a[b}) \omega_{c]} + \mathcal{L}_\ell g_{a[b} N_{c]d}\,. $$
    \end{itemize}
\end{lemma}

\begin{proof}
    Choosing coordinates $(u,x^1,x^2)$ where $\ell = \partial_u$, without loss of generality, we may write any such Ehresmann connection as
    \begin{equation} 
        \omega = du + \omega_i (u,x,y) dx^i,\quad i =1,2. 
    \end{equation} 
    As the degenerate metric can be treated as a Riemannian signature metric on a 2-manifold, it is always possible to diagonalize this with an upper-triangular coframe:
     \begin{equation} \begin{aligned}
       m^1 &= m^1_{~1}(u, x^i) dx^1\,, \\
       m^2 &= m^2_{~1}(u,x^i) dx^1 + m^2_{~2}(u,x^i) dx^2.
    \end{aligned} \end{equation}
    Choosing $(\theta^\ai) = (\omega, m^1, m^2)$ as a coframe, the dual frame $(e_\ai)$ is then 
    \begin{equation} \begin{aligned}
        e_1 &= \ell = \partial_u, \\
        e_2 &= \frac{\omega_2 m^2_{~1} - \omega_1 m^2_{~2}}{m^1_{~1} m^2_{~2}} \partial_u + \frac{1}{m^1_{~1}} \partial_{x^1} - \frac{m^2_{~1}}{m^1_{~1} m^2_{~2}} \partial_{x^2}, \\
        e_3 &= - \frac{\omega_2}{m^2_{~2}} \partial_{u} + \frac{1}{m^2_{~2}} \partial_{x^1} .\\ 
    \end{aligned} \end{equation}
    In the frame basis, we can constrain the connection coefficients $\Gamma^\ai_{~\bb\cc}$ for $\nabla$ using the hypothesis. The conditions $\nabla \ell = \nabla g  =0$ yield purely algebraic expressions:
    \begin{equation}
    \begin{aligned}
        & \nabla \ell = 0 \rightarrow  \Gamma^{1}_{~1\ai}= \Gamma^{2}_{~1\ai}=\Gamma^{3}_{~1\ai} = 0, \quad \ai = 1,2,3, \\
        &\nabla g= 0 \rightarrow \Gamma^{2}_{~2\ai}=\Gamma^{3}_{~3\ai}=0, \Gamma^{3}_{~21} = -\Gamma^{2}_{~31}, \Gamma^{3}_{~22} = -\Gamma^{2}_{~32},\Gamma^{3}_{~23} = -\Gamma^{2}_{~33} 
    \end{aligned}
    \end{equation}
    giving
    \begin{equation} \begin{aligned}
        \nabla e_1 =& 0, \\
        \nabla e_2 =& \Gamma^1_{~21} e_1 \otimes \theta^1 + \Gamma^1_{~22} e_1 \otimes \theta^2 + \Gamma^1_{~23} e_1 \otimes \theta^3 \\ 
        & -\Gamma^2_{31} e_3 \otimes \theta^1 - \Gamma^2_{~32} e_3 \otimes \theta^2 - \Gamma^2_{~33} e_3 \otimes \theta^3 \\
        \nabla e_3 =& \Gamma^1_{~31} e_1 \otimes \theta^1 + \Gamma^1_{~32} e_1 \otimes \theta^2 + \Gamma^1_{~33} e_1 \otimes \theta^3  \\ 
        &+ \Gamma^2_{31} e_2 \otimes \theta^1 + \Gamma^2_{~32} e_2 \otimes \theta^2 + \Gamma^2_{~33} e_2 \otimes \theta^3.
    \end{aligned} \label{eqn:Thm32a} \end{equation}
    \noindent The covariant derivative of the Ehresmann connection is 
    \begin{equation} \begin{aligned}
        \nabla \omega = \nabla \theta^1 =& -\Gamma^1_{21} \theta^2 \otimes \theta^1 - \Gamma^1_{22} \theta^2 \otimes \theta^2 - \Gamma^1_{23} \theta^2 \otimes \theta^3 \\ 
        &- \Gamma^1_{31} \theta^3 \otimes \theta^1 - \Gamma^1_{32} \theta^3 \otimes \theta^2 - \Gamma^1_{33} \theta^3 \otimes \theta^3.  
    \end{aligned} \label{eqn:Thm32b} \end{equation}
    
    By hypothesis, $\omega$ is such that the torsion of $\nabla$ is minimal with respect to $(M,g,\ell,\omega)$. According to Definition~\ref{minimal-torsion-definition}, we may further determine $\Gamma$, using
    \begin{equation} \label{lie-derivs}
    \begin{aligned}
        \mathcal{L}_{\ell}g &= (\ln m^1_{~1})_{,u} \theta^2 \theta^2 + (\ln m^2_{~2})_{,u} \theta^3 \theta^3 + \\
        &\quad \quad 2 \left(\frac{(m^2_{~1})_{,u} m^2_{~2} - (m^2_{~2})_{,u} m^2_{~1}}{m^1_{~1} m^2_{~2}} \right) \theta^2 \theta^3, \\
        \mathcal{L}_\ell \omega &= \frac{\omega_{1,u} m^2_{~2} - \omega_{2,u} m^2_{~1}}{m^1_{~1} m^2_{~2}} \theta^2 + \frac{\omega_{2,u}}{m^2_{~2}} \theta^3.
    \end{aligned}
    \end{equation}
    Computing $T^\ai_{~\bb \cc}\ell^\cc = T(\ell, \cdot)$ using Cartan's first structure equation, we find
    \begin{equation}
    \begin{aligned}
        T(\ell,\cdot) &= \left( \Gamma^1_ {~21} + \frac{ \omega_{1,u} m^2_{~2}-\omega_{2,u}m^2_{~1}}{m^1_{~1} m^2_{~2}} \right) e_1 \otimes \theta^2 + \left(  \Gamma^1_ {~31} +\frac{\omega_{2,u}}{m^2_{~2}} \right) e_1 \otimes \theta^3 \\
        &\quad \quad  (\ln m^1_{~1})_{,u} e_2 \otimes \theta^2 + (\ln m^2_{~2})_{,u} e_3 \otimes \theta^3 + \Gamma^2_{~31} e_2 \otimes \theta^3 \\
        &\quad \quad -\left( \Gamma^2_{~31} +\frac{(m^2_{~1})_{,u} m^2_{~2} - (m^2_{~2})_{,u} m^2_{~1}}{m^1_{~1} m^2_{~2}} \right) e_2 \otimes \theta^3.
    \end{aligned}
    \end{equation}
    Combining the above with Definition~\ref{minimal-torsion-definition},
    we have the following expressions for connection coefficients: 
    \begin{equation}
    \begin{aligned}
        \Gamma^1_{~21} =& 0 \\ 
        \Gamma^1_{~31} =& 0, \\
         \Gamma^2_ {~31} =& -\left(\frac{m^2_{~2,u} m^2_{~1}- m^2_{~1,u} m^2_{~2} }{m^1_{~1} m^2_{~2}}\right), \\
         \Gamma^1_{~23} =& \Gamma^1_{~32} + \frac{ -\omega_{2,u} \omega_1 + \omega_{1,u} \omega_2  + \omega_{2,1}-\omega_{1,2}}{m^1_{~1} m^2_{~2}}, \\
         \Gamma^2_{~32} = & \frac{(m^1_{~1})_{,u} \omega_2 - (m^1_{~1})_{,2}}{m^1_{~1} m^2_{~2}}, \\
       \Gamma^2_{~33} = & - \left(\frac{(m^2_{~1})_{,u} \omega_2 - (m^2_{~2})_{,u} \omega_1 + (m^2_{~2})_{,1}-(m^2_{~1})_{,2}}{m^1_{~1} m^2_{~2}} \right).
    \end{aligned} \label{eqn:Thm32c}
    \end{equation}
    
    In light of this, to prove the first two conditions in the theorem, we express the torsion tensor explicitly: 
        \begin{equation}
    \begin{aligned}
        {\bf T} &= 2 \left( \frac{ \omega_{1,u} m^2_{~2}-\omega_{2,u}m^2_{~1}}{m^1_{~1} m^2_{~2}} \right) e_1 \otimes \theta^1 \wedge \theta^2 + 2\left( \frac{\omega_{2,u}}{m^2_{~2}} \right) e_1 \otimes \theta^1 \wedge \theta^3 \\ 
        &\quad + 2(\ln m^1_{~1})_{,u} e_2 \otimes \theta^1 \wedge \theta^2 + 2(\ln m^2_{~2})_{,u} e_3\otimes \theta^1  \wedge \theta^3  \\
        &\quad  + \quad 4 \left(\frac{(m^2_{~1})_{,u} m^2_{~2} - (m^2_{~2})_{,u} m^2_{~1}}{m^1_{~1} m^2_{~2}} \right) (e_3 \otimes \theta^1 \wedge \theta^2 + e_2 \otimes \theta^1 \wedge \theta^3).
    \end{aligned}
    \end{equation}

\medskip

    Now first suppose that the trace of the torsion tensor is non-vanishing. Then, taking its trace we find
    \begin{equation}
    \begin{aligned}
            V_\ai :=&  T^\bb_{~\bb \ai} = \left( \frac{ \omega_{1,u} m^2_{~2}-\omega_{2,u}m^2_{~1}}{m^1_{~1} m^2_{~2}} \right) \theta^2 + \left( \frac{\omega_{2,u}}{m^2_{~2}} \right) \theta^3 + (\ln m^1_{~1} m^2_{~2} )_{,u} \theta^1.
    \end{aligned}
    \end{equation}
    Comparing with $\mathcal{L}_{\ell} \omega$, we find that
    $$V_\ai - \mathcal{L}_{\ell} \omega_\ai = (\ln m^1_1 m^2_2)_{,u} \omega\,.$$
    Now suppose that $V(\ell) \neq 0$. Then, it follows that 
    \begin{equation}
        \tfrac{1}{V(\ell)} (V - \mathcal{L}_\ell \omega) = \omega\,.
    \end{equation}
    \noindent It thus follows that if
    $$\nabla_{(\ai} \gamma_{\bb)} =: N_{\ai \bb}\,,$$
    where
    $\gamma := \tfrac{1}{V(\ell)} (V - \mathcal{L}_\ell \omega)$, then $\nabla \odot \omega = N$, as required.

    On the other hand,
    if the trace of the torsion tensor is horizontal, meaning $V(\ell) = 0$, so that $V = \mathcal{L}_\ell \omega$, or the torsion tensor is trace-free, we may use the fact that the torsion tensor with all indices lowered can be written as 
    \begin{equation}
        T_{\ai \bb \cc} = (\mathcal{L}_\ell g_{\ai[\bb}) \omega_{\cc]}
    \end{equation}
    \noindent Taking the covariant derivative of this tensor and imposing the second condition in the theorem statement  yields $$\Gamma^1_{~22} = N_{22}, \quad \Gamma^1_{~33} = N_{33} \text{ and } \Gamma^1_{~32} = N_{23} +  \frac{\omega_{2,u} \omega_1- \omega_{1,u} \omega_2 - \omega_{2,x}+\omega_{1,y}}{2 m^1_{~1} m^2_{~2}}$$ from which it follows that $ \nabla \odot\omega = N$.

    In either case, it is easy to check by way of Equation~\ref{eqn:Thm32b} that $N$ is horizontal, meaning that $N(\ell,\cdot) = 0$.
    
    Proving the opposite direction is straightforward.

    \end{proof}

    \begin{remark}
        Note that Lemma~\ref{nonvanishing-torsion-lemma} may be generalized to higher (fixed) dimensions in a straightforward way, however the computations quickly become unwieldy.
    \end{remark}

    We are now equipped to study the conditions required to produce special Carrollian manifolds and potential Carroll structures.

\section{Minimal features of special Carrollian manifolds}\label{sec:SCMs}

Special Carrollian manifolds first appeared (implicitly) in the literature in~\cite{bekaert2018connections}; there, a larger class of manifold structures were specified in the form of a principal Ehresmann connection and what was called there a torsional special Carrollian connection, constructed in such a way that the Ehresmann connection is parallel. Indeed, in that work, it was shown that when given any principal Ehresmann connection, a torsional special Carrollian connection is determined uniquely if the connection is constructed in such a way that the only torsion present arises from the Lie derivative of the metric---hence, the name ``minimal" torsion. However, the reverse implication is not provided in their work, nor is it trivially obtained from it.

\subsection{Principal 1-form implies connection} \label{sec:SCM-1form-con}
The minimal data that one may in principle specify for a Carrollian structure is an Ehresmann connection. However, it is clear from~\cite{bekaert2018connections} that this is insufficient to obtain a unique affine connection. However, given any such Ehresmann connection, one may transform it into a principal Ehresmann connection using Carrollian boosts. This is captured in the following lemma.
\begin{lemma}
Let $(M,g, \ell)$ be a Carrollian structure equipped with an Ehresmann connection, $\nu'$. Then $\nu'$ can always be transformed via a Carrollian boost $\nu' \to \nu$  such that $\nu$ is principal, i.e., $\mathcal{L}_{\ell} \nu = 0$.
\end{lemma}

\begin{proof}
    We will choose a coordinate system $(t, x^i)$, $i,j,k \in \{ 1,\ldots, n\}$ where $\ell = \partial_t$. The Ehresmann connection can be expressed as
    \begin{equation}
        \nu' = dt + \alpha_i(t, x^j) dx^i. \nonumber
    \end{equation}
    Then if $\mathcal{L}_{\ell} \nu' \neq 0$ this is equivalent to 
    \begin{equation}
        \partial_t \alpha_i \neq 0.
    \end{equation}
    Using a Carrollian boost with parameters $B_i = -\alpha_i(t,x^j) + b_i(t,x^j)$ the new Ehresmann connection $$ \nu = dt + b_i(t,x^j) dx^i$$ is now principal. 
\end{proof}

Having established that any Carrollian structure may be equipped with a principal Ehresmann connection, we may now establish a result equivalent to that of~\cite[Proposition A.21]{bekaert2018connections}.

\begin{theorem} \label{thm:preBM}
    Let $(M, g, \ell)$ be a Carrollian structure equipped with a principal Ehresmann connection, $\nu$ such that $\nu(\ell)=1$. Then, there exists a unique torsionful connection $\nabla$ such that $(M,g,\ell,\nu,\nabla)$ is a special Carrollian manifold.
\end{theorem}

\begin{proof}
    We work in coordinates $x^a = (t,x^i)$ such that $\ell = \partial_t$, and we will prove the theorem by establishing that in these coordinates, there exists unique Christoffel symbols that satisfy the hypothesis. Note that in these coordinates, the principal Ehresmann connection may be expressed as
    $$\nu = dt + \beta_i(t,x^j) dx^i\,.$$

    As we require that the torsion is minimal (per Definition~\ref{def:SCM}), which according to Lemma~\ref{minimal-lemma} makes it uniquely specified, we need only check that the symmetric parts of the Christoffel symbols $\Gamma^a_{(bc)}$ are uniquely fixed by the other constraints making a Carrollian structure a special Carrollian manifold: that $\nabla \ell = \nabla g = \nabla \nu = 0$ and the torsion of $\nabla$ is minimal.

    In these coordinates, the condition $\nabla \ell = 0$ is equivalent to
    \begin{align} \label{first-connect-SCM}
        \Gamma^a_{bt} = 0\,,
    \end{align}
    which implicitly fixes $\Gamma^a_{(bt)}$.

Now consider the purely ``spatial" part of the covariant derivative of the metric:
\begin{align*} \nabla_{i}g_{jk} & = \partial_i g_{jk} - \Gamma^a_{ij}g_{ak} - \Gamma^a_{ik}g_{aj} \\
&= \partial_i g_{jk} - \Gamma^l_{ij}g_{lk} - \Gamma^l_{ik}g_{lj},
\end{align*}
where in the second line we again used $g_{tb} = 0$. Now $\nabla_i g_{jk} = 0$ is identical to the compatibility condition between the \emph{invertible} metric $g_{ij}$ and the purely `spatial' connection whose Christoffel symbols are $\Gamma^i_{jk}$. Hence, standard arguments of Riemannian geometry dictate that $\nabla_i g_{jk} = 0$ if and only if
\begin{equation} \label{Levi-Civita-SCM}
\Gamma^i_{(jk)} = \frac{1}{2}g^{lj}(\partial_i g_{kl} + \partial_k g_{li} - \partial_l g_{ik}).
\end{equation}
This fixes the purely horizontal component of the connection.

Finally, from the condition that $\nabla \nu = 0$, we have 
\begin{align} 
\partial_{(a}\nu_{b)} - \Gamma^c_{(ab)}\nu_c  = 0 \label{last-connectEq1-SCM} \\
\implies \Gamma^t_{(ab)}\nu_t + \Gamma^i_{(ab)}\nu_i  = \partial_{(a}\nu_{b)}  \nonumber \\
\implies  \Gamma^t_{(ab)}  = - \Gamma^i_{(ab)}\nu_i + \partial_{(a}\nu_{b)}, \label{last connect-SCM}
\end{align}
where in the last line we used the fact that $1 = \nu_a\ell^a = \nu_t$. Note that Equation~\ref{first-connect-SCM} and Equation~\ref{Levi-Civita-SCM} ensure that $\Gamma^i_{(ab)}$ is fixed in Equation~\ref{last connect-SCM}, and so $\Gamma^t_{(ab)}$ is determined by Equation~\ref{last connect-SCM}.

All that remains to check is that Equation~\ref{last connect-SCM} is compatible with the parallel vector condition $\Gamma^a_{bt} = 0$. By setting $b = t$ in Equation~\ref{last-connectEq1-SCM} and using the facts that $\nu_t = 1$, $\Gamma^c_{at} = 0$, and $g_{at} = 0$, we find that
$$\partial_t \nu_a = \Gamma^c_{ta} \nu_c\,.$$
Using that 
$\Gamma^c_{at} = 0$, we have that $\Gamma^c_{ta} = T^c_{ta} = \ell^b T^c_{ba}$, so
$$\partial_t \nu_a = \ell^b T^c_{ba} \nu_c\,.$$
But because $\nu$ is principal (meaning $\mathcal{L}_{\ell} \nu = 0$), it follows that $\partial_t \nu_a = 0$. So the above expression reduces to $\ell^b T^c_{ba} \nu_c = 0$, which is consistent with Constraint~\ref{minimal-constraint}, which we have already assumed holds by hypothesis.

So Equations~\ref{first-connect-SCM},~\ref{Levi-Civita-SCM}, and~\ref{last connect-SCM}, together with the condition that the torsion is minimal, uniquely specify the connection and are mutually consistent. It follows that $(M,g,\ell,\nu,\nabla)$ is a special Carrollian manifold.

\end{proof}

\subsection{Connection implies 1-form} \label{sec:SCM-con-1form}

As connections have far more degrees of freedom than 1-forms, it is not surprising that a unique compatible connection for a given principal Ehresmann connection can be constructed as in the previous subsection by imposing geometric constraints. On the other hand, the converse is not, in general, true. Rather, strong geometric restrictions must be placed on the connection to ensure that such a parallel Ehresmann connection exists. This is captured by the following theorem.

\begin{theorem}
\label{Thm:BMcharacterization}
    Let $(M^3,g,\ell)$ be a Carrollian structure equipped with a linear connection $\nabla$ satisfying $\nabla \ell = 0 = \nabla g$.
    Suppose that there exists an Ehresmann connection $\nu$ such that the torsion of $\nabla$ is minimal with respect to the tuple $(M,g,\ell,\nu)$. Then, $(M,g,\ell,\nu,\nabla)$ is a special Carrollian manifold if and only if one of the following identities holds for $X, Y, Z \in \Gamma(TM)$:
        \begin{itemize}
        \item $\nu(T(\ell, X)) = 0$, and $V = Trace(T) \neq 0$,
        $$\nabla_{(a}\gamma_{b)} = 0, \qquad \gamma_b = \frac{V_b}{V(\ell)} $$
        \item $\nu(T(\ell, X)) = 0$, and $V = 0$, 
        $$-\nabla_d T_{abc} = \nabla_d (\mathcal{L}_\ell g_{a[b}) \nu_{c]}  $$
        \item $T = 0$, $\nu(R(X,Y) Z ) = 0$, and 
        $$ (\nabla_e R^a_{~bcd}) \nu_a  = 0.$$
    \end{itemize}
\end{theorem}

\begin{proof}

We begin by considering the case where the torsion of $\nabla$ is non-vanishing. In that case, following Lemma~\ref{nonvanishing-torsion-lemma} with $N = 0$, it suffices to show that $\mathcal{L}_{\ell} \nu = 0$. However, using Definition~\ref{minimal-torsion-definition}, we see that if $\nu(T(\ell,X)) = 0$, then $\mathcal{L}_{\ell} \nu = 0$. But these are conditions required for the theorem when the torsion is non-vanishing, so this case is resolved.

It thus suffices to show that when the third identity holds, $(M,g,\ell,\nu,\nabla)$ is a special Carrollian manifold, and that if $(M,g,\ell,\alpha,\nabla)$ is a special Carrollian manifold with vanishing torsion, then the third condition follows. We do so using the notation and computations found in the proof of Lemma~\ref{nonvanishing-torsion-lemma}.

First, note that because the torsion vanishes, it follows that $\mathcal{L}_{\ell} \nu = 0$ and $\mathcal{L}_{\ell} g = 0$, and both the spatial components of the coframe and the Ehresmann connection $\nu$ are independent of $u$.  Thus, from Equations~\ref{lie-derivs} and~\ref{eqn:Thm32c}, it follows that, in that choice of frame,
    \begin{equation}
    \begin{aligned}
        \Gamma^{B}_{~1A} =& 0, \\
        \Gamma^{2}_{~2A} =& 0, \\
        \Gamma^{3}_{~3A} =& 0, \\
        \Gamma^1_{~21} =& 0, \\ 
        \Gamma^1_{~31} =& 0, \\
         \Gamma^2_ {~31} =& 0, \\
         \Gamma^1_{~23} =& \Gamma^1_{~32} + \frac{\nu_{2,1}-\nu_{1,2}}{m^1_{~1} m^2_{~2}}, \\
         \Gamma^2_{~32} =& -\Gamma^3_{~22} = -\frac{  (m^1_{~1})_{,2}}{m^1_{~1} m^2_{~2}}, \\
       \Gamma^2_{~33} = & -\Gamma^3_{~23} = - \frac{ (m^2_{~2})_{,1}-(m^2_{~1})_{,2}}{m^1_{~1} m^2_{~2}}.
    \end{aligned} \label{eqn:Thm34-coeffs}
    \end{equation}
The only connection coefficients left unfixed here are $\Gamma^1_{~\ii \jj}$, where $\ii,\jj = 2,3$.
Now we note that $R^\ai_{~\bb \cc \dd}\nu_\ai = 0$ by hypothesis. That is, in the choice of frame, we have that $R^{1}_{~\bb\cc\dd} = 0$ but perhaps $R^\ii_{~\bb\cc\dd} \neq 0$ for $\ii = 2,3$.

Now by hypothesis, $\nu(R(X,\ell) Y) = 0$. Taking note of the connection coefficients above, we have that
$$\Gamma^1_{~\ii \jj,u} e_1 \otimes \theta^\ii \otimes \theta^\jj\,,$$
i.e. $\Gamma^1_{~\ii\jj}$ is independent of the $u$ coordinate. Now, examining the covariant derivative of the curvature tensor:
    \begin{equation}
    \begin{aligned}
        R^\ai_{~\bb\cc\dd;\eee}\nu_\ai =& \Gamma^1_{~i\eee} R^\ii_{~\bb\cc\dd} - \Gamma^\ff_{~\bb\eee} R^1_{~\ff\cc\dd} - \Gamma^\ff_{~\cc\eee} R^1_{~\bb\ff\dd} - \Gamma^\ff_{~\dd\eee} R^1_{~\bb\cc\ff} \\ 
        & = \Gamma^1_{~i\eee} R^\ii_{~\bb\cc\dd}.  
    \end{aligned}
    \end{equation}
    \noindent If this vanishes, then it follows that $\Gamma^1_{~31}=\Gamma^1_{~21} = \Gamma^1_{~23} = \Gamma^1_{~32}=\Gamma^1_{~22}=\Gamma^1_{~33}=0$, since the sole algebraically independent component of the curvature tensor is $R^2_{~323}$. If this component vanishes then the transverse space reduces to the Euclidean space and the result trivially holds. \color{black}

    On the other hand, if we assume that $(M,g,\ell,\nu,\nabla)$ is a special Carrollian manifold with $\nabla$ torsion-free, then we may express
    $$R^\ai{}_{~\dd\bb\cc} \nu_\ai = -2 \nabla_{[\bb} \nabla_{\cc]} \nu_\dd\,.$$
    But as $\nu$ is parallel by assumption, this vanishes. Hence, the third constraint is trivially satisfied.

\end{proof}

\begin{remark}
    We note that while the proof of this theorem has been presented for a 3-dimensional special Carrollian geometry, this proof can naturally be extended to higher dimensions with two minor modifications. The first modification is choosing two horizontal vector-fields for which $g(Y, T(\ell, X)) \neq 0$ in the case where the torsion is trace-free in order to construct $e_2$ and $e_3$. The second modification arises in the torsion-free subcase where it may be necessary to take higher order covariant derivatives of the curvature tensor in order to exhaust the conditions for the connection coefficients in $\nabla \theta^1$.  
\end{remark}

\begin{remark}
    As Theorem~\ref{Thm:BMcharacterization} is a necessary and sufficient condition for the existence of an Ehresmann connection that satisfies the conditions required to produce a special Carrollian manifold, it is clear that, given an affine connection, there does not always exist an Ehresmann connection that makes $(M,g,\ell,\nabla)$ into a special Carrollian manifold.
\end{remark}

While the above theorem heavily constrains the space of possible Ehresmann connections for a given affine connection so that the tuple $(M,g,\ell,\nu,\nabla)$ is a special Carrollian manifold, it is not clear that $\nu$ is unique. The following theorem shows that it is not always unique.

\begin{theorem}
    Let $(M,g,\ell, \nu \nabla)$ be a Carrollian manifold. The Ehresmann connection, $\nu$, is unique if and only if the spatial metric, $g$, does not admit a vorticity-free Killing vector field. 
\end{theorem}

\begin{proof}
    In local coordinates, where $\nu = dt'$, suppose that $\tilde{\nu} = dt' + \beta_i dx^i$ is some other parallel Ehresmann connection. Then the condition $\nabla \tilde{\nu} = 0$ is equivalent to $\nabla (\beta_i dx^i) = 0$. Taking the symmetric and anti-symmetric parts of this rank 2 tensor equation, we find that the spatial part of $\tilde{\nu}$ satisfies the Killing equations and is vorticity free, since $d(\beta_i dx^i ) =0$. From this proof the opposite direction follows trivially. 
\end{proof}

\section{Minimal Features of Potential Carroll Structures}\label{sec:PCSs}

Having resolved the question of geometric constraints required to minimally-specify a special Carrollian manifold, we now turn to the same question for potential Carroll structures.

\subsection{1-form implies connection}
As in Subsection~\ref{sec:SCM-1form-con}, the existence of a connection satisfying a particular constraint equation given the data of a Carrollian structure and a 1-form is generally underdetermined. As such, one might expect that, with sufficiently strong constraints, any 1-form will naturally determine a connection that makes the 1-form a metric potential. This expectation is captured in the following theorem.

\begin{theorem}
Let $(M,g,\ell)$ be a Carollian structure equipped with an Ehresmann connection $\alpha$ satisfying $\alpha(\ell) = 1$. Then there exists a unique torsionful connection $\nabla$ such that $(M,g,\ell,\alpha,\nabla)$ is a potential Carroll structure.
\end{theorem}
\begin{proof}
This proof closely follows that of Theorem~\ref{thm:preBM}.

We work in coordinates $x^a = (t,x^i)$ such that $\ell = \partial_t$. We will construct the connection in these coordinates by constructing the Christoffel symbols. From Definition~\ref{def:PCS}, we impose several constraints: that $\nabla \ell = 0 = \nabla g$, that $\nabla \odot \alpha = g$, and that the torsion of $\nabla$ is minimal. Note from Lemma~\ref{minimal-lemma} that the minimal torsion is fixed once a 1-form $\alpha$ is specified, so we need only fix the symmetric components of the Christoffel symbols. We will do so by demanding that $\nabla \ell = 0 = \nabla g$ and that $\nabla \odot \alpha = g$.

In the specified coordinates, the condition $\nabla \ell = 0$ is equivalent to \begin{align} \label{first-connect}
    \Gamma^a_{bt} = 0\,.
\end{align}
Note that because the torsion is fixed, Equation~\ref{first-connect} implicitly fixes $\Gamma^a_{(bt)}$.

Now consider the purely ``spatial" part of the covariant derivative of the metric:
\begin{align*} \nabla_{i}g_{jk} & = \partial_i g_{jk} - \Gamma^a_{ij}g_{ak} - \Gamma^a_{ik}g_{aj} \\
&= \partial_i g_{jk} - \Gamma^l_{ij}g_{lk} - \Gamma^l_{ik}g_{lj},
\end{align*}
where in the second line we again used $g_{tb} = 0$. Now $\nabla_i g_{jk} = 0$ is identical to the compatibility condition between the \emph{invertible} metric $g_{ij}$ and the purely `spatial' connection whose Christoffel symbols are $\Gamma^i_{jk}$. Hence, standard arguments of pseudo-Riemannian geometry dictate that $\nabla_i g_{jk} = 0$ if and only if
\begin{equation} \label{Levi-Civita}
\Gamma^i_{(jk)} = \frac{1}{2}g^{lj}(\partial_i g_{kl} + \partial_k g_{li} - \partial_l g_{ik}).
\end{equation}
This fixes the purely horizontal component of the connection, as the torsion (which picks out the antisymmetric piece) is fixed.

Finally, from the potential equation $\nabla_{(a}\alpha_{b)}  = g_{ab}$, we have 
\begin{align} 
\partial_{(a}\alpha_{b)} - \Gamma^c_{(ab)}\alpha_c  = g_{ab} \label{last-connectEq1} \\
\implies \Gamma^t_{(ab)}\alpha_t + \Gamma^i_{(ab)}\alpha_i  = \partial_{(a}\alpha_{b)} - g_{ab} \nonumber \\
\implies  \Gamma^t_{(ab)}  = - \Gamma^i_{(ab)}\alpha_i + \partial_{(a}\alpha_{b)} - g_{ab}, \label{last connect}
\end{align}
where in the last line we used the fact that $1 = \alpha_i\ell^i = \alpha_t$. Note that Equation~\ref{first-connect} and Equation~\ref{Levi-Civita} ensure that $\Gamma^i_{(ab)}$ is fixed in Equation~\ref{last connect}, and so $\Gamma^t_{(ab)}$ is determined by Equation~\ref{last connect}.

All that remains to check is that Equation~\ref{last connect} is compatible with the parallel vector condition $\Gamma^a_{bt} = 0$. By setting $b = t$ in Equation~\ref{last-connectEq1} and using the facts that $\alpha_t = 1$, $\Gamma^c_{at} = 0$, and $g_{at} = 0$, we find that
$$\partial_t \alpha_a = \Gamma^c_{ta} \alpha_c \,.$$
But because $\Gamma^c_{at} = 0$, we have that $\Gamma^c_{ta} = T^c_{ta} = \ell^b T^c_{ba}$, so that
$$\partial_t \alpha_a = \ell^b T^c_{ba} \alpha_c\,.$$
But in these coordinates, this is equivalent to Constraint~\ref{minimal-constraint}, which we have already assumed holds by hypothesis.

So Equations~\ref{first-connect},~\ref{Levi-Civita}, and~\ref{last connect}, together with the condition that the torsion is minimal, uniquely specify the connection and are mutually consistent. It follows that $(M,g,\ell,\alpha,\nabla)$ is a potential Carroll structure.
\end{proof}

\subsection{Connection implies 1-form}
As in Subsection~\ref{sec:SCM-con-1form}, it is clear that significant geometric constraints must be placed on a Carrollian structure to ensure that there exists a 1-form that is a metric potential. This fact is captured in the following theorem.

\begin{theorem} \label{Thm:PCSconnection}
    Let $(M^3,g,\ell)$ be a Carrollian structure equipped with a linear connection $\nabla$ satisfying $\nabla \ell = 0 = \nabla g$.
    Suppose that there exists an Ehresmann connection $\alpha$ such that the torsion of $\nabla$ is minimal with respect to the tuple $(M,g,\ell,\alpha)$. Then $(M,g,\ell,\alpha,\nabla)$ is a potential Carroll structure if and only if there exists an orthogonal frame $\{\ell,e_2,e_3\}$, where $\alpha(e_2) = \alpha(e_3) = 0$, and one of the following identities holds:
        \begin{itemize}
        \item $V = Trace(T) \neq 0$ and $V(\ell) \neq 0$, 
        $$\nabla_{(a}\gamma_{b)} = g_{ab}, \qquad \gamma_b = \frac{V_b-\mathcal{L}_\ell \alpha_b}{V(\ell)} $$
        \item $V(\ell) = 0$
        $$-\nabla_d T_{abc} = \nabla_d (\mathcal{L}_\ell g_{a[b}) \alpha_{c]} + \mathcal{L}_\ell g_{a[b} g_{c]d}, $$
        \item $T = 0$, $\alpha(R(\cdot,\ell) \cdot ) = 0$, and 
        $$ \alpha_a e_2^{~b} e_3^{~d} \nabla_{(f} \left[ R^a_{~|b|c)d} - \ell^a R^e_{~|b|c)d}\alpha_e \right]  = X g_{fc}, $$
    \end{itemize}
    where $X$ is a scalar function.
\end{theorem}

\begin{proof}

We begin by considering the case where the torsion of $\nabla$ is non-vanishing. In that case, using Lemma~\ref{nonvanishing-torsion-lemma} with $N = g$, the first two conditions follow immediately. Hence, all that is left is to consider the case where torsion vanishes.

    If the torsion tensor vanishes identically, then $\mathcal{L}_\ell \alpha = 0$ and $\mathcal{L}_\ell g = 0$. It follows that the connection coefficients in equation \eqref{eqn:Thm32c} simplify significantly, but $\Gamma^1_{~23}$ and $\Gamma^1_{~32}$ are still non-zero. This contributes to vertical contributions to the curvature tensor. However, using the definition of the curvature tensor, it follows that
    \begin{equation}
        R^\ai_{~\bb\cc\dd} \ell^\cc \alpha_\ai = \Gamma^1_{~\ii\jj,u} e_1 \otimes \theta^\ii \otimes \theta^\jj  
    \end{equation}
    \noindent Hence, imposing the vanishing of this condition requires that $\Gamma^1_ {~\ii\jj}$ must be independent of the $u$ coordinate.
    
    To derive algebraic conditions on the necessary connection coefficients required to prove that $ \nabla \odot\alpha = g$ follows from the last condition, we consider the purely horizontal part of the curvature tensor: 
    \begin{equation}
        \tilde{R}^\ai_{~\bb\cc\dd} := R^\ai_{~\bb\cc\dd} - \ell^\ai R^\eee_{~\bb\cc\dd} \alpha_\eee  
    \end{equation}
    This ensures that $\tilde{R}^\ai_{~\bb\cc\dd} \alpha_\ai = 0$ and so
    \begin{equation}
    \begin{aligned}
        \tilde{R}^\ai_{~\bb\cc\dd;\ff}\alpha_\ai =  \Gamma^1_{~\ii\ff} \tilde{R}^\ii_{~\bb\cc\dd}. \label{eqn:PotentialCurvatureDerivative} 
    \end{aligned}
    \end{equation}
    Computing the curvature tensor using equation \eqref{eqn:Thm32c} with $\mathcal{L}_\ell g = 0$, it is straightforward to show that $\tilde{R}^\ai_{~\bb\cc\dd} = R^\ii_{~\jj\kk\li}$, that is, the curvature tensor of the horizontal space arising from the Levi-Civita connection $\Gamma^\ii_{~\jj\kk}$. Imposing the condition $$ \Gamma^1_{~I(\ff}\tilde{R}^\ii_{~|\bb|\cc)\dd}e_2^\bb e_3^\dd = X g_{\ff\cc} $$ yields 3 linear equations for 3 unknowns $\Gamma^1_{~\jj\kk},~ (\jj \geq \kk)$ whose solution give $$\Gamma^1_{~22} = \Gamma^1_{~33} = -1 \text{ and } \Gamma^1_{~32} = \frac{- \alpha_{2,x}+\alpha_{1,y}}{2 m^1_{~1} m^2_{~2}}$$ from which it follows that $ \nabla \odot\alpha = g$. 

    As the only condition required to be a potential Carroll structure not granted by hypothesis is that $\nabla \odot \alpha = g$, and we have established that this identity holds for all of the cases considered, the proof is complete.

    In order to prove the other direction, we need only consider the case where torsion vanishes identically. A potential Carrollian structure always permits a coframe where 
    \begin{equation}
        \begin{aligned}
            \nabla \theta^1 &= - \theta^2 \theta^2 - \theta^3 \theta ^3 - d \theta^1, \\
            \nabla \theta^\ii &= \Gamma^\ii_{~\jj\kk} \theta^\jj \theta^\kk,
        \end{aligned}
    \end{equation}
    \noindent where $\Gamma^\ii_{\jj\kk}$ is the Levi-Civita connection of the transverse space. Using the definition of the curvature tensor \eqref{eqn:GeomTensors} the above condition on the connection coefficients yields the only two non-zero components of the curvature tensor
    \begin{equation} \label{nonzero-curvatures}
        R^1_{~\ii\ii\jj}, \text{ and } R^\ii_{~\jj\kk\li}, 
    \end{equation}
    \noindent from which it follows that $$\alpha(R(\cdot, \ell)\cdot) = 0. $$ (Note that the repeated indices in $R^1{}_{IIJ}$ in Equation~\ref{nonzero-curvatures} does not imply Einstein summation.)
    Computing $\tilde{R}^\ai_{~\bb\cc\dd} = R^\ai_{~\bb\cc\dd} - \ell^\ai R^\eee_{~\bb\cc\dd}\alpha_\eee$ and taking the covariant derivative, the argument following equation \eqref{eqn:PotentialCurvatureDerivative} applies and the connection coefficients in $\nabla \theta^1$ already satisfy these linear equations and hence  $$ \Gamma^1_{~\ii(\ff}\tilde{R}^\ii_{~|\bb|\cc)\dd}e_2^\bb e_3^\dd = X g_{\ff\cc}. $$ where the scalar X takes the form    $$ X = - \Gamma^1_{~\ii\jj}R^{\ii~\jj}_{~2~3}.$$
\end{proof}

\begin{remark}
    While this proof has been presented in three dimensions, it is not difficult to extend this to the $d$-dimensional case by changing the last condition to hold for all pairs of horizontal frame vectors.  
\end{remark}

\begin{remark}
    Just in the case with Theorem~\ref{Thm:BMcharacterization}, from Theorem~\ref{Thm:PCSconnection}, it follows that even given a connection on a Carrollian structure, it is not guaranteed that there exists an Ehresmann connection $\alpha$ that makes the tuple $(M,g,\ell,\alpha,\nabla)$ a potential Carroll structure.
\end{remark}

\section{Relating special Carrollian manifolds to potential Carroll structures} \label{sec:relations}

It should be evident from the results of Sections~\ref{sec:SCMs} and~\ref{sec:PCSs} that special Carrollian manifolds and potential Carroll structures are closely related. However, a priori, it is not clear when a potential Carroll structure admits, by solely modifying the Ehresmann connection, a special Carrollian manifold. Similarly, it is not clear when a special Carrollian manifold admits, by solely modifying the Ehresmann connection, a potential Carroll structure.

In this section we prove two theorems that show how heavily restricted the space of special Carrollian manifolds $(M,g,\ell,\alpha,\nabla)$  that admit a potential Carroll structure $(M,g,\ell,\alpha+\theta,\nabla)$ is and vice versa.

\subsection{From potential Carroll structures to SCMs}
We begin by showing that the existence of a modified Ehresmann connection turning a potential Carroll structure into a special Carrollian manifold heavily restricts the behavior of the scalar curvature on horizontal embedded surfaces in $M^3$.
\begin{theorem} \label{thm:PCS-to-SCMs}
    Let $(M^3,g,\ell,\alpha,\nabla)$ be a potential Carroll structure and suppose there exists $\theta \in \Gamma(T^* M)$ such that $\nu:= \alpha + \theta$ satisfies $\nabla \nu = 0$, $\nu(\ell) = 1$, and $T^a_{bc} \nu_a = 0$. Furthermore, suppose that $\ker \nu$ is integrable. Then, let $\iota : \Sigma \hookrightarrow M$ determined by $T \Sigma = \ker \nu$ be any horizontal embedded surface in $M$ with induced metric $\bar{g} := \iota^*g$ and induced Levi-Civita connection $\bar{\nabla}$. Finally, define $B_{ab} = \bar{g} - \bar{\nabla}_{(a} (\iota^* \alpha)_{b)}$. It follows that:
    \begin{enumerate}
        \item if $(\Sigma, \bar{g})$ has vanishing scalar curvature, then $\iota^* d \alpha$ is a constant multiple of the volume form on $(\Sigma,\bar{g})$;
        \item if $(\Sigma,\bar{g})$ has everywhere non-vanishing scalar curvature, then
        $$2 \bar{\nabla}_a (\bar{Sc}^{-1} \bar{\nabla}_{[c} B_{b]}{}^c) - B_{ab} = 0\,. $$
    \end{enumerate}
\end{theorem}

\begin{proof}

    This proof relies on a characterization of tensors that can be expressed as the Hessian of some function as given by Robert Bryant~\cite{BryantSO}.

    Because $\nu(\ell) = 1$, it follows that $\theta$ is horizontal, namely if $H^*M := \operatorname{Ann}(\ell)$, then we require that  $\theta \in \Gamma(H^*M)$.

Because $\nabla_{[a} \nu_{b]} = 0$ and $T^a_{bc} \nu_a = 0$, it follows that
$$d \theta = - d \alpha\,,$$
which implies that, at least locally,
$$\theta = -\alpha + df$$
for some function $f \in C^\infty M$, and so $\nu = df$. Furthermore, note that because $\theta \in \Gamma(H^* M)$, it must be the case that $df(\ell) = 1$. Now in a neighborhood around a point $p \in \Sigma \subset M$, we may choose coordinates $(t,x,y)$ such that $\ell = \partial_t$ and $\partial_{x^i} \in \Gamma(HM)$ so that $\bar{g} = e^{2 \Omega(x,y)}(dx^2 + dy^2)$. In these coordinates, we have that
$$f = t + h(x,y)\,,$$
and so
$$\nu = dt + dh(x,y)\,.$$
Thus,
$$dh = \alpha + \theta - dt\,.$$

Now consider the horizontal component of $\bar{\nabla}_{(a}(\iota^* \alpha)_{b)}$. A straightforward computation shows that, because $ \nabla \odot\alpha = g$ and $dt = \nu - dh$ is the vertical component of $\alpha$, it follows that
$$\bar{\nabla}^2 h = B\,.$$

Differentiating this equation and skewing, we find that
\begin{align} \label{curv-constr}
    \bar{R}_{abcd} \bar{\nabla}^d h = \bar{\nabla}_{[a} B_{b]c}\,.
\end{align}
Now because we are in two dimensions, we have that
$$\bar{R}_{abcd} = \frac{\bar{Sc}}{2}(\bar{g}_{ac} \bar{g}_{bd} - \bar{g}_{ad} \bar{g}_{bc})\,.$$

Now consider the case where $\bar{Sc} = 0$. In that case, the Riemann curvature of $(\Sigma, \iota^* g)$ vanishes and hence $\bar{g}$ is locally isometric to the flat metric. As we have already chosen isothermal coordinates, we find that up to constant rescalings, $\bar{g} = dx^2 + dy^2$. So,
$$\partial_{[a} B_{b]c} = 0\,.$$
Expanding this expression and noting that $\bar{\nabla} (\bar{g}) = 0$, one finds that
$$\partial_c [\partial_a (\iota^*\alpha)_b - \partial_b (\iota^* \alpha)_a] = 0\,.$$
Thus, it follows that
$$\partial_x (\iota^* \alpha)_y - \partial_y (\iota^* \alpha)_x = c\,,$$
for some constant $c$, and so
$$2 d(\iota^* \alpha) = c dx \wedge dy\,,$$
and so in particular,
$$2 \iota^* d \alpha = c dx \wedge dy\,,$$
the volume form for $(\Sigma,\bar{g})$.

On the other hand, suppose that $\bar{Sc}$ is nowhere vanishing. Then tracing over the curvature condition in Equation~\ref{curv-constr}, we have that
$$\frac{\bar{Sc}}{2} \bar{\nabla}_b h= \bar{\nabla}_{[a} B_{b]}{}^a\,.$$
Because $\bar{Sc}$ is nowhere vanishing, we can divide through by $\bar{Sc}$ and differentiate again, noting that $\bar{\nabla}^2 h = B$, yielding
$$2 \bar{\nabla}_a (\bar{Sc}^{-1} \nabla_{[c} B_{b]}{}^c) - B_{ab} = 0\,.$$
This completes the proof.
    
\end{proof}

\begin{remark}
    Clearly, for sufficiently generic $\alpha$ and $\nabla$, the scalar curvature on any horizontal embedded surface $\Sigma$ will not obey any differential equations, nor will $d \alpha$ be a constant multiple of a volume form. As such, Theorem~\ref{thm:PCS-to-SCMs} establishes that only non-generic potential Carroll structures may be modified to obtain a special Carrollian manifold.
\end{remark}

\subsection{From SCMs to potential Carroll structures}
Following Theorem~\ref{thm:PCS-to-SCMs}, we will also see that the existence of a modification of the Ehresmann connection turning a special Carrollian manifold into a potential Carroll structure is highly restrictive. In particular, we find that the existence of such a modification implies that horizontal embedded hypersurfaces have non-isometric homotheties.
\begin{theorem} \label{conf-killing-constraint}
    Let $(M^d,g,\ell,\nu,\nabla)$ be a special Carrollian manifold with a spatial hypersurface embedding $\iota: \Sigma \hookrightarrow M$ satisfying $T\Sigma = \operatorname{ker}\nu$. Then, there exists a 1-form $\theta \in \Gamma(T^* M)$ such that $\alpha := \nu + \theta$ satisfies $\alpha(\ell) = 1$ and
    $$ \nabla_{(a} \alpha_{b)} = g_{ab}$$
    if and only if $\iota^*g$ has a non-isometric homothetic vector field.
\end{theorem}

\begin{proof}
    First, suppose that there exists $\theta \in \Gamma(T^* M)$ such that $\alpha := \nu + \theta$ is a 1-form potential for the metric and is an Ehresmann connection. In that case, because $\nu$ is parallel, we have that
    $$\nabla_{(a} \theta_{b)} = g_{ab}\,.$$

    Furthermore, because $\nu(\ell) = \alpha(\ell) = 1$, we have that $\theta(\ell) = 0$; in other words, $\theta$ is horizontal and hence can be pulled back to the spatial hypersurface, $\bar{\theta} := \iota^* \theta$, without modification.

    It also follows, given the connection for a special Carrollian manifold, we have that $\iota^*(\nabla \theta) = \bar{\nabla} \bar{\theta}$, where $\bar{\nabla}$ is the Levi-Civita connection on $T\Sigma$. Writing $\bar{g} := \iota^* g$, we therefore have that
    \begin{align} \label{cke-spatial}
    \bar{\nabla}_{(a} \bar{\theta}_{b)} = \bar{g}_{ab}\,.
    \end{align}
    This is the equation for a homothetic vector field $\bar{g}^{ab} \theta_b$ with constant $1$, rather than $0$, as is the case for an isometry.

    The other direction is evident by considering the contrapositive and examining the proof above. Clearly, if there is no such 1-form $\alpha$, it follows easily that there is no $\bar{\theta}$ satisfying Equation~(\ref{cke-spatial}). Thus the theorem holds.   
\end{proof}

In general, Riemannian manifolds do not have homotheties, let alone non-isometric homotheties. To illustrate how restrictive this condition is, we consider the compact orientable case.
\begin{corollary}
    Let $(M,g,\ell,\nu,\nabla)$ be a special Carrollian manifold with a compact submanifold $\iota : \Sigma \hookrightarrow M$ satisfying $T\Sigma := \ker \nu$. Then, there exists no 1-form $\theta \in \Gamma(T^* M)$ such that $\alpha := \nu + \theta$ satisfies $\alpha(\ell) = 1$ and
    $$\nabla_{(a} \alpha_{b)} = g_{ab}\,.$$
\end{corollary}
\begin{proof}
    From Theorem~\ref{conf-killing-constraint}, it suffices to check whether there exists a non-isometric homothety on $(\Sigma,\iota^* g)$. But as $\Sigma$ is compact, it is well-known~\cite{Obata1970} that the only homotheties on $(\Sigma,\iota^* g)$ are isometries. This completes the proof.
\end{proof}

\section*{Acknowledgements}
S.B. was partially supported by the Operational Programme Research Development and Education Project No. CZ.02.01.01/00/22-010/0007541. S.B. would also like to acknowledge the generous financial support of Matt Gimlim.

\bibliographystyle{ieeetr}
\bibliography{bib}

\end{document}